\documentclass[10pt,twoside]{amsart}

\usepackage{amssymb}
\usepackage{verbatim}
\usepackage[T1]{fontenc}
 \usepackage{pifont}
 \usepackage{color}

\numberwithin{equation}{section} 

\newtheorem{theorem}{Theorem}[section] 
\newtheorem{lemma}[theorem]{Lemma}
\newtheorem{proposition}[theorem]{Proposition}
\newtheorem{corollary}[theorem]{Corollary}

\theoremstyle{definition}
\newtheorem{definition}[theorem]{Definition}
\newtheorem{remark}[theorem]{Remark}
\newtheorem{example}[theorem]{Example}

\newtheorem{notation}[theorem]{Notation}

\newcommand{\popo}{\mathbb{P}^1 \times \mathbb{P}^1}

\newcommand{\N}{\mathbb{N}}

\begin{document}


\title[The interpolation problem for a set of three fat points in $\popo$]{The interpolation problem for a set of three fat points in $\popo$}

\author{Giuseppe Favacchio}
\address{Dipartimento di Matematica e Informatica\\
Viale A. Doria, 6 - 95100 - Catania, Italy} \email{favacchio@dmi.unict.it} \urladdr{/}

\author{Elena Guardo}
\address{Dipartimento di Matematica e Informatica\\
Viale A. Doria, 6 - 95100 - Catania, Italy}
\email{guardo@dmi.unict.it}
\urladdr{http://www.dmi.unict.it/$\sim$guardo/}

\keywords{interpolation problem; multiprojective spaces; Hilbert functions; fat points}
\subjclass[2000]{13F20, 13A15, 13D40, 14M05}
\thanks{ January 13, 2017}

\begin{abstract}This is an appendix to the recent paper of Favacchio and Guardo. In  these notes we describe explicitly a minimal bigraded free resolution and the bigraded Hilbert function of a set of 3 fat points whose support is an almost complete intersection (ACI) in $\popo.$   This solve the interpolation problem for three points with an ACI support.
\end{abstract}

\maketitle


\renewcommand{\thetheorem}{\thesection.\arabic{theorem}}
\setcounter{theorem}{0}
\section{Introduction}
The interpolation problem is an interesting problem which involve many different fields of research included algebraic geometry and commutative algebra. It asks to determine the dimension of the vector space of homogeneous polynomials of each given degree that vanish on a set of points under some condition.  
Let $R := k[ \mathbb{P}^n] = k[x_0,\ldots,x_n ]$ be the standard polynomial ring over a infinite field, consider the homogeneous ideal $$I_{Z} := \bigcap_{i=1}^{s}I_{P_i}^{m_i}$$  of $R,$ where $I_{P_i}$ is the homogeneous ideal associated to $P_i$ and $m_i$ are non negative integers.  The interpolation problem then asks what can be the Hilbert function of $R/I_{Z},$
where the Hilbert function $$H_Z(t) := dim_k (R/I_{Z})_t,\ \text{ for all}\ t>0$$ computes the dimension of the
homogeneous component $(R/I_{Z})_t$ of degree $t$ of $R/I_{Z}$ for all $t \in \mathbb{N} $. 
Geramita-Maroscia-Roberts \cite{GMR1983} and Geramita-Gregory-Roberts \cite{GGR1986} give an answer for set of reduced points, i.e. $m_i = 1$ for all $i$. 
When some of the multiplicities are greater than one, the interpolation problem remains open.


A variation of the interpolation problem was introduced by Giuffrida-Maggioni-Ragusa \cite{GMR1992} changing the ambient space from a single projective space $\mathbb{P}^n$ to $\mathbb{P}^{n_1}\times \cdots\times\mathbb{P}^{n_r}$ a multiprojective space. 
One of the main difference is that the coordinate ring of a collection of points in a multiprojective space is not always Cohen-Macaulay.
The multigraded version of the interpolation problem is still open for reduced set of points in $\popo.$ 
A partial answer was given in $\popo,$ was given by Giuffrida-Maggioni-Ragusa \cite{GMR1992}, for arithmetically Cohen-Macaulay reduced sets of points, and by Guardo-Van Tuyl \cite{GuVTbook} for arithmetically Cohen-Macaulay  sets of fat points. Recently, the first author (see \cite{F}) gave a description of the Hilbert functions for bigraded algebras in $k[\popo].$ 

Without the hypothesis of Cohen-Macaulayness, we are still far away from a complete understanding of the problem and any result in this direction could be interesting. In \cite{FGu} the authors recursively computed the minimal free resolution of a (non homogeneous) set $Z$ of fat points whose support is an almost complete intersection (ACI), that in general are not arithmetically Cohen-Macaulay(ACM) even if the support is ACM.   

The aim of this work is to describe the Hilbert function, and the bigraded Betti numbers, for a set $Z$ of three fat points in $\popo$ on an ACI support $X$ i.e. $I_X$ is generated by three bihomogeneous forms. In Section \ref{sec:recall} we recall some results and notation from \cite{FGu}.  In Section \ref{sec:numerical} we introduce and study a numerical function which will be related with the homological invariant of such sets of points.
Finally, in Section \ref{sec:Betti} first we give a formula to compute the graded Betti numbers and the Hilbert Function of a set of three fat points on an ACI support, then we give an answer to the interpolation problem for those set of points.

\section{Notation and preliminary result}\label{sec:recall}
Throughout this paper $R:=k[x_0,x_1,x_2,x_3]$ is the coordinate ring of $\popo$ over an infinite field of characteristic $0,$ with the bigrading given by $\deg x_0 =\deg x_1 = (1,0)$ and $\deg x_2 = \deg x_3 = (0,1)$.
Let $H_i$ be horizontal lines of type $(1,0)$ and $V_j$ vertical lines of type $(0,1)$, then  points in $\popo$ can be denoted by $P_{ij}:=H_i\times V_j.$ With an abuse of notation we denote the ideal $I_{P_{ij}}=(H_i,V_j).$ 

Given a set of distinct points $X$ and positive integers $m_{ij}$, we call $Z=\sum_{P_{ij}\in X} m_{ij}P_{ij}$ a set of fat points supported at $X$. The associated ideal to $Z$ is $I_Z:=\bigcap_{P_{ij}\in X}I_{P_{ij}}^{m_{ij}}.$

In these note we use the following notation as in \cite{FGu}. 
\begin{notation}\label{notation} We set $Z:=m_{11}P_{11}+m_{12}P_{12}+m_{21}P_{21},$ where $m_{ij}\ge 0$ and, without loss of generality, $m_{12}\ge m_{21}.$ We denote by $Z_1=(m_{11}-1)_+P_{11}+m_{12}P_{12}+(m_{21}-1)_+P_{21},$ where $(n)_+:=\max\{n,0\}.$
\end{notation} 

The following results were proven in \cite{FGu} in a more general setting. For the convenience of reader, we recall them in a version which is useful to our focus.   
\begin{lemma}[Lemma 2.2, \cite{FGu}]\label{L2.2}Let $Z=m_{11}P_{11}+m_{12}P_{12}$
	be a set of two collinear fat points. Set $M:=\max\{m_{11},m_{12}\}$, then a minimal free resolution of $I_Z$ is
	{\Small $$0\to \bigoplus_{t=1}^{M} R(-t, -(m_{11}-t+1)_+-(m_{12}-t+1)_+)  \to  \bigoplus_{t=0}^{M} R(t, -(m_{11}-t)_+-(m_{12}-t)_+) \to I_{Z} \to 0$$}
\end{lemma}
\begin{lemma}[Lemma 3.4, \cite{FGu}]\label{L3.4}Let $Z:= m_{12}P_{12}+ m_{21}P_{21}$ be a set of two non collinear fat points. Then a minimal free resolution of $I_{Z}$ is
	$$0\to \bigoplus_{(a,b,c,d)\in \mathcal{D}_2} R(-a-b, -c-d)\to \bigoplus_{(a,b,c,d)\in \mathcal{D}_1} R(-a-b, -c-d) \to $$
	$$ \to \bigoplus_{(a,b,c,d)\in \mathcal{D}_0} R(-a-b, -c-d) \to I_{Z} \to 0$$
	where:\\
	$\mathcal{D}_0:=\{(a,b,c,d) |\ 0\le a,d \le m_{12},\ 0\le b,c \le m_{21},\ a+d=m_{12},\ b+c=m_{21}  \} $\\
	$\mathcal{D}_1:=\{(a,b,c,d) |\ 0\le a,d \le m_{12},\ 0\le b,c \le m_{21},\ (a+d=m_{12}+1,\ b+c=m_{21})\vee( a+d=m_{12},\ b+c=m_{21}+1 )\}$\\
	$\mathcal{D}_2:= \{(a,b,c,d) |\ 0\le a,d \le m_{12},\ 0\le b,c \le m_{21},\ a+d=m_{12}+1,\ b+c=m_{21}+1  \}.$
\end{lemma}
These two lemmas describe the resolution of $Z$ in the degenerating case when one of the multiplicities is $0.$
The next result allows to recursively compute the resolution of $Z$ in the remaining cases.  
\begin{lemma}[Remark 2.10,Theorem 2.12, \cite{FGu}]\label{R2.10+T2.12} Using Notation \ref{notation}, let $0\to
	L_2\to L_1\to L_0$ be a minimal free resolution of $I_{Z_1},$ 
	then a minimal free resolution for $I_{Z}$ is {\Small
		$$ 0\to \bigoplus_{(a,b)\in A_2(Z)}R(-a,-b)\oplus L_2(0,-1) \to$$
		$$\to  \bigoplus_{(a,b)\in A_1(Z)}R(-a,-b)^2
		\oplus R(-m_{11}-m_{21},-(m_{12}-m_{11})_+-1  )
		\oplus L_1(0,-1)\to $$
		\begin{equation}\label{res}\to \bigoplus_{(a,b)\in A_0(Z)}R(-a,-b)\oplus L_0(0,-1) \to I_{Z}\to 0 \end{equation}
	}
	where
	{\Small
		\begin{tabular}{l}
			$A_0(Z):=\{(a,b)\ |\ a+b= m_{11}+m_{21}+ (m_{12}-m_{11})_+\ \ \text{and}\ \ 0\le b\le (m_{12}-m_{11})_+  \}$\\
			$A_1(Z):=\{(a,b)\ |\ a+b=1+ m_{11}+m_{21}+ (m_{12}-m_{11})_+\ \text{and}\ 1 \le b\le (m_{12}-m_{11})_+ \}$\\
			$A_2(Z):=\{(a,b)\ |\ a+b=2+ m_{11}+m_{21}+ (m_{12}-m_{11})_+\ \text{and}\ 2 \le b\le (m_{12}-m_{11})_+ +1 \}.$\\
		\end{tabular}}
\end{lemma}	 
In Section \ref{sec:Betti} we give a closed formula to compute the bigraded Betti numbers of $Z$.

\section{Numerical facts}\label{sec:numerical}
In order to determine the graded Betti numbers of $Z$ we need to introduce a numerical function depending on a parameter $t\in \mathbb{Z}.$ We define inductively the function $\varphi_t:\mathbb{Z}\to\mathbb{Z}$ as follows:

$$\varphi_1(n)=\begin{cases} 1 & \text{if}\ n = 0\\
                            0 & \text{otherwise}
                             \end{cases}$$

and, for $t> 1$

$$\varphi_t(n)=\begin{cases}\varphi_{t-1}(n) & \text{if}\ 0\le n < t-1\\
                           \varphi_{t-1}(n)+1 & \text{if}\ t-1\le n < 2t-1\\
                            0 & \text{otherwise.}
                            \end{cases}$$

We will use the convention $\varphi_t(n)=0$ if $t\le 0.$   
\begin{remark}One can inductively check that
$\varphi_t(n)=\left(\left\lfloor \frac{\min\{n,\ 2t-2-n\}}{2}
\right\rfloor+1\right)_+.$ But, for our purposes, we prefer the use of the recursive definition. 
\end{remark}
Let $t,d\in \mathbb{Z}$ be two integers such that $t\ge d,$ we
define the function $\varphi_{t,d}(n):\mathbb{Z}\to\mathbb{Z}$ in the following way:
     $$\varphi_{t,d}(n)=\varphi_{t}(n+d)-\varphi_{d}(n+d).$$

We give an example in order to clarify the notation.
\begin{example} To shorten the notation, we represent the functions as tuples, where the first entry is their value in $0$, the second entry is their value in $1$ and so on.

\begin{tabular}{cl}
$\varphi_{1}$&$=(1,0,0,\ldots)$\\
$\varphi_{2}$&$=(1,1,1,0,0,\ldots)$\\
$\varphi_{3}$&$=(1,1,2,1,1,0,\ldots)$\\
$\varphi_{4}$&$=(1,1,2,2,2,1,1,0,\ldots)$\\
$\cdots$&\\
$\varphi_{7}$&$=(1,1,2,2,3,3,4,3,3,2,2,1,1,0,\ldots)$\\
$\varphi_{7}-\varphi_{4}$&$=(0,0,0,0,1,2,3,3,3,2,2,1,1,0,\ldots)$\\
$\varphi_{7,4}$&$=(1,2,3,3,3,2,2,1,1,0,\ldots)$\\
$\varphi_{4,-3}$&$=(0,0,0,1,1,2,2,2,1,1,0,\ldots)$\\
\end{tabular}
\end{example}
We have the following property.
\begin{proposition}\label{indphi} Let $t\ge d$ be two integers, then:

$\varphi_{t,d+1}(n-1)=\begin{cases} \varphi_{t,d}(n)-1 & \text{if}\ 0\le n \le d\\
                                    \varphi_{t,d}(n) & \text{otherwise.}
                                    \end{cases}$

\end{proposition}
\begin{proof} By definition, we have
$\varphi_{t,d+1}(n-1)= \varphi_{t}(n+d)-\varphi_{d+1}(n+d),$ hence

$\varphi_{t,d+1}(n-1)=\begin{cases}\varphi_{t}(n+d)-\varphi_{d}(n+d) & \text{if}\ 0\le n+d < d\\
                                    \varphi_{t}(n+d)-\varphi_{d}(n+d)-1 & \text{if}\ d\le n+d < 2d+1\\
                                    \varphi_{t}(n+d) -0 & \text{otherwise}
                                    \end{cases}$

      \small$=\begin{cases}         \varphi_{t}(n+d)-   \varphi_{d}(n+d) & \text{if}\ 0\le n+d < d\\
                                    \varphi_{t}(n+d)-       \varphi_{d}(n+d)-1 & \text{if}\ d\le n+d < 2d+1\\
                                    \varphi_{t}(n+d) -0 & \text{otherwise}
                                    \end{cases}$

           $=\begin{cases}          \varphi_{t}(n+d)-       \varphi_{d}(n+d)-1 & \text{if}\ 0\le n \le d\\
                                    \varphi_{t}(n+d)-\varphi_{d}(n+d) & \text{otherwise}
                                    \end{cases}.$
\end{proof}

\section{The graded Betti numbers of $I_Z$}\label{sec:Betti}
Let $X$ be a set of fat points in $\popo$ and let $I_X\subseteq R := k[\popo]$ be the bihomogeneous ideal associated to $X.$   Then we can associate to $I_X$ a minimal bigraded free resolution of the form
$$0\to \bigoplus R(-i,-j)^{\beta_{2,(i,j)}(X)}\to \bigoplus  R(-i,-j)^{\beta_{1,(i,j)}(X)}\to \bigoplus  R(-i,-j)^{\beta_{0,(i,j)}(X)}\to R\to R/I_X \to 0 $$
where $R(-i,−j)$ is the free $R$-module obtained by shifting the degrees of $R$
by $(i,j).$ The graded Betti number $\beta_{u,(i,j)}(X)$ of $R/I_X$ counts the number of a
minimal set of generators of degree $(i,j)$ in the $u$-th syzygy module of $R/I_X$.

Using the same strategy as in \cite{FGu} we split the description in two cases.

\subsection{First case $m_{11}\le m_{21}$}

\begin{theorem}\label{Thmcaso1} With the Notation \ref{notation}, if $m_{11}\le m_{21}$ then the bigraded Betti numbers of $I_Z$ are:
\Small\begin{tabular}{l} $\beta_{0,(a,b)}(Z)$=$\begin{cases}
                    \left(\min\{a,b-m_{11}, m_{21}-m_{11}\}+1\right)_+ +\varphi_{m_{12},m_{12}-m_{11}}(b)&\text{if}\ a+b=m_{21}+m_{12}\\
                    0 & \text{otherwise}
                    \end{cases}$\\

\\
$\beta_{1,(a,b)}(Z)$=$\left( \beta_{0,(a,b-1)}(Z) +\beta_{0,(a-1,b)}(Z) -1 \right)_+$\\
$\beta_{2,(a,b)}(Z)$=$\left( \beta_{0,(a-1,b-1)}(Z) -1 \right)_+$
\end{tabular}

\end{theorem}
\begin{proof}We proceed by induction on $m_{11}.$ If $m_{11}=0$ then $\varphi_{m_{12},m_{12}}(b)=0$ and statement is true by Lemma \ref{L3.4}. Assume $m_{11}>0,$ by Lemma \ref{R2.10+T2.12} we get
$$\beta_{0,(a,b)}(Z)=\begin{cases} \beta_{0,(a,b-1)}({Z_1})+1 & \text{if}\ a+b=m_{12}+ m_{21} \text{and}\ b\le m_{12}-m_{11}   \\
                                \beta_{0,(a,b-1)}({Z_1}) & \text{if}\ a+b=m_{12}+ m_{21}  \text{and}\ b> m_{12}-m_{11}  \\
                                0 & \text{otherwise}
                                \end{cases}$$

Thus,  set $S:=m_{12}+ m_{21}$ and $B:=
m_{12}-m_{11},$ we have
{\Small$$ \beta_{0,(a,b)}(Z)=\begin{cases} \left(\min\{a,b-m_{11}, m_{21}-m_{11}\}+1\right)_+ +\varphi_{m_{12},B+1}(b-1)+1 & \text{if}\ a+b=S\  \text{and}\ b\le B  \\
 \left(\min\{a,b-m_{11}, m_{21}-m_{11}\}+1\right)_+ +\varphi_{m_{12},B+1}(b-1) & \text{if}\ a+b=S\  \text{and}\ b> B  \\
                                0 & \text{otherwise}
                                \end{cases}$$}
and by using Proposition \ref{indphi} it is
{\Small$$\beta_{0,(a,b)}(Z) =\begin{cases} \left(\min\{a,b-m_{11}, m_{21}-m_{11}\}+1\right)_+ +\varphi_{m_{12},B}(b) & \text{if}\ a+b=S\  \text{and}\ b\le B  \\
 \left(\min\{a,b-m_{11}, m_{21}-m_{11}\}+1\right)_+ +\varphi_{m_{12},B}(b) & \text{if}\ a+b=S\  \text{and}\ b> B  \\
                                0 & \text{otherwise.}
                                \end{cases}$$}

The computation of $\beta_{1,(a,b)}(Z)$ also follows by induction and Lemma \ref{R2.10+T2.12}.

\begin{tabular}{l}
	$\beta_{1,(a,b)}({Z})=\begin{cases}  \beta_{1,(a,b-1)}({Z_1})+2 & \text{if}\ a+b-1= m_{21}+ m_{12}\ \text{and}\ 1\le b\le m_{12}-m_{11} \\ 
	\beta_{1,(a,b)}({Z})+1 & \text{if}\ (a,b)= (m_{11}+m_{21}, m_{12}-m_{11}+1) \\ 
	0 &  \text{otherwise} 
	\end{cases}$\\
	$=\begin{cases}  \beta_{0,(a,b-2)}({Z_1})+\beta_{0,(a-1,b-1)}({Z_1})+1 & \text{if}\ a+b-1= m_{21}+ m_{12}\ \text{and}\ 1\le b\le m_{12}-m_{11} \\ 
	\beta_{0,(a,b-2)}({Z_1})+\beta_{0,(a-1,b-1)}({Z_1}) & \text{if}\ (a,b)= (m_{11}-m_{21}, m_{12}-m_{11}+1) \\ 
	0 &  \text{otherwise} 
	\end{cases}$\\
	$=\begin{cases}  \beta_{0,(a,b-1)}({Z})+\beta_{0,(a-1,b)}({Z})-1 & \text{if}\ a+b-1= m_{21}+ m_{12}\ \text{and}\ 1\le b\le m_{12}-m_{11} \\ 
	\beta_{0,(a,b-1)}({Z})+\beta_{0,(a-1,b)}({Z})-1 & \text{if}\ (a,b)= (m_{11}+m_{21}, m_{12}-m_{11}+1) \\ 
	0 &  \text{otherwise} 
	\end{cases}$\\
\end{tabular}

The computation of $\beta_{2,(a,b)}(Z)$ requires the same procedure as above  by using the inductive hypotheses and
Lemma \ref{R2.10+T2.12}.
\end{proof}

We show in the following example how to compute the graded Betti numbers of a set of three fat points using Theorem \ref{Thmcaso1}.

\begin{example}\label{exstart1}  Consider $Z=2P_{11}+5P_{12}+4P_{21},$ to compute $\beta_{0,(a,b)}(Z)$ we first need to compute the bigraded Betti numbers of  $Z':=(2-2)P_{11}+5P_{12}+(4-2)P_{21}=5P_{12}+2P_{21}.$ By Lemma \ref{L3.4}, the non zero bigraded Betti numbers of $R/I_{Z'}$ are:

\begin{tabular}{lll}
$\beta_{0,(7,0)}(I_{Z'})=1$&$\beta_{1,(7,1)}(I_{Z'})=2$&$\beta_{2,(7,2)}(I_{Z'})=1$\\
$\beta_{0,(6,1)}(I_{Z'})=2$&$\beta_{1,(6,2)}(I_{Z'})=4$&$\beta_{2,(6,3)}(I_{Z'})=2$\\
$\beta_{0,(5,2)}(I_{Z'})=3$&$\beta_{1,(5,3)}(I_{Z'})=5$&$\beta_{2,(5,4)}(I_{Z'})=2$\\
$\beta_{0,(4,3)}(I_{Z'})=3$&$\beta_{1,(4,4)}(I_{Z'})=5$&$\beta_{2,(4,5)}(I_{Z'})=2$\\
$\beta_{0,(3,4)}(I_{Z'})=3$&$\beta_{1,(3,5)}(I_{Z'})=5$&$\beta_{2,(3,6)}(I_{Z'})=2$\\
$\beta_{0,(2,5)}(I_{Z'})=3$&$\beta_{1,(2,6)}(I_{Z'})=4$&$\beta_{2,(2,7)}(I_{Z'})=1$\\
$\beta_{0,(1,6)}(I_{Z'})=2$&$\beta_{1,(1,7)}(I_{Z'})=2$&\\
$\beta_{0,(0,7)}(I_{Z'})=1$&&\\
\end{tabular}
\\
Moreover we have $\varphi_{5,3}=(1,2,2,2,1,1,0\ldots).$ Hence, if
$a+b=9$ we have
$$\beta_{0,(a,b)}(Z)=\beta_{0,(a,b-2)}(Z')+\varphi_{5,3}(b).$$
Then the non zero bigraded Betti numbers of $R/I_{Z}$ are:

\begin{tabular}{lll}
$\beta_{0,(9,0)}(Z)=1$      &$\beta_{1,(9,1)}(Z)=2$       &$\beta_{2,(9,2)}(Z)=1$\\
$\beta_{0,(8,1)}(Z)=2$      &$\beta_{1,(8,2)}(Z)=4$       &$\beta_{2,(8,3)}(Z)=2$\\
$\beta_{0,(7,2)}(Z)=3$      &$\beta_{1,(7,3)}(Z)=6$       &$\beta_{2,(7,4)}(Z)=3$\\
$\beta_{0,(6,3)}(Z)=4$      &$\beta_{1,(6,4)}(Z)=7$       &$\beta_{2,(6,5)}(Z)=3$\\
$\beta_{0,(5,4)}(Z)=4$      &$\beta_{1,(5,5)}(Z)=7$       &$\beta_{2,(5,6)}(Z)=3$\\
$\beta_{0,(4,5)}(Z)=4$      &$\beta_{1,(4,6)}(Z)=6$       &$\beta_{2,(4,7)}(Z)=2$\\
$\beta_{0,(3,6)}(Z)=3$          &$\beta_{1,(3,7)}(Z)=5$     &$\beta_{2,(3,8)}(Z)=2$\\
$\beta_{0,(2,7)}(Z)=3$            &$\beta_{1,(2,8)}(Z)=4$       &$\beta_{2,(2,9)}(Z)=1$\\
$\beta_{0,(1,8)}(Z)=2$            &$\beta_{1,(1,9)}(Z)=2$       &\\
$\beta_{0,(0,9)}(Z)=1.$            &&
\end{tabular}

\end{example}






\subsection{Second Case $m_{11}> m_{21}$} To conclude the description of the graded Betti numbers of $R/I_Z$ we need some preliminaries.

\begin{definition}\label{defD} Let $Z=m_{11}P_{11}+m_{12}P_{12}+m_{21}P_{21}$ be a set of three fat points in $\popo$, set $B_Z:=m_{12}-m_{11},$ we define the following sets of integers associated to $Z$:
	
\noindent{
\begin{tabular}{l}
$D_1(Z)=\{(a,b)\in\mathbb{N}^2\ | \ 0\le b < (-B_Z)_+-(-B_Z-m_{21})_+\ \text{and}\    a+2b=m_{11}+m_{21} \}$\\
$D_2(Z)=\{(a,b)\in\mathbb{N}^2\ | \ \ (-B_Z)_+-(-B_Z-m_{21})_+\le b<m_{21}\ \text{and}\     a+b=\max\{m_{11}, m_{12}+m_{21}\}  \}$\\
$D_3(Z)=\{(a,b)\in\mathbb{N}^2\ | \ \ m_{21}\le b \le m_{21}+|B_Z+m_{21}| \ \text{and}\ a+b=\max\{m_{11}, m_{12}+m_{21} \}$\\
$D_4(Z)=\{(a,b)\in\mathbb{N}^2\ | \ b > m_{21}+|B_Z+m_{21}| \
\text{and}\  2a+b=m_{11}+m_{12}   \}.$
\end{tabular}}

\end{definition}

Some immediate remarks follows from Definition \ref{defD}. We recall that $Z$ is an arithmetically Cohen-Macaulay (ACM for short) set of points if $R/I_Z$ is Cohen Macaulay. See \cite{GuVTbook} for more details about ACM set of fat points.
\begin{remark}\label{remD}   It is a matter of computation to show that
\begin{enumerate}
  \item[i)]  $B_{Z_1}=B_{Z}+1;$
    \item[ii)]  $D_1(Z)=\emptyset$ if and only if $B_{Z}\ge 0;$
    \item[iii)]  $D_2(Z)=\emptyset$ if and only if  $B+m_{21}\le 1$ (in this case $Z$ is ACM by Theorem 6.21 \cite{GuVTbook} );
    \item[iv)] If $B<0$ and  $B+m_{21}= 1$ then $(a,-B)\notin D_3(Z)$ for any $a;$
    \item[v)] If $(a,b-1)\in D_i(Z_1)$ then $(a,b)\in D_i(Z),$ for $i=1, 2, 3, 4;$
    \item[vi)] If $(a,b)\in D_i(Z)$ for some $i$ then $(a,\bar{b} ), (\bar{a} ,b)\notin D_j$ for any $\bar{a}\neq a,\ \bar{b}\neq b$ and $j=1, 2, 3, 4.$
    \item [vii)] If $(a-1,b)\in D_4(Z)$ then $(a,b-1)\notin D_3(Z).$
  \end{enumerate}
  
\end{remark}
\begin{theorem}\label{Thmcaso2} If $m_{11}> m_{21}$  then the bigraded Betti numbers of $R/I_Z$ are:

{\Small \begin{tabular}{l} $\beta_{0,(a,b)}(Z)$=$\begin{cases}
                       1  & \text{if}\ (a,b)\in D_1(Z)\\
                       \varphi_{m_{21}+B,B}(b)   & \text{if}\ (a,b)\in D_2(Z)\\
                       1+\varphi_{m_{21}+B,B}(b)      & \text{if}\ (a,b)\in D_3(Z)\\
                       1    & \text{if}\ (a,b)\in D_4(Z)\\
                       0 & \text{otherwise}   \end{cases}$\\

\\
$\beta_{1,(a,b)}(Z)=\begin{cases}
            1 & \textit{if} \ (a,b-1)\in D_1(Z)\\
            \beta_{0,(a,b-1)}(Z) +\beta_{0,(a-1,b)}(Z) -1 & \textit{if} \ (a,b-1)\in D_2(Z)\cup D_3(Z)\\
            1 & \textit{if} \ (a-1,b)\in D_4(Z)\\
            0 & \textit{otherwise}  \end{cases}$\\
\\
$\beta_{2,(a,b)}(Z)= \begin{cases}
                   \beta_{0,(a-1,b-1)}(Z) -1 &  \textit{if} \ (a-1,b-1)\in D_2(Z)\cup D_3(Z)\\
                     0 & \textit{otherwise}\end{cases}$\\
\end{tabular}}

where the $D_i(Z)$ are defined in Definition \ref{defD}.
\end{theorem}
\begin{proof}We proceed by induction on $m_{21}.$ If $m_{21}=0$ then $Z$ is a set of 2 collinear (fat) points, $D_1=D_2=\emptyset$ and $\varphi_{B,B}(b)=0.$ Therefore the statement follows by Lemma \ref{L2.2}. Assume now $m_{21}>0,$ by Lemma \ref{R2.10+T2.12} we have

\begin{tabular}{l}

$\beta_{0,(a,b)}(Z) =\begin{cases} \beta_{0,(a,b-1)}({Z_1})+1
& \text{if}\ a+b= m_{11}+m_{21}+ (m_{12}-m_{11})_+ \\ & \
\text{and}\ b\le (m_{12}-m_{11})_+   \\
\beta_{0,(a,b-1)}({Z_1}) &  \text{otherwise.}
                                \end{cases}$
\end{tabular}

If $m_{12}\le m_{11},$ i.e.  $B_Z< 0$ we get

\begin{tabular}{l}
$\beta_{0,(a,b)}(Z) =\begin{cases} 1 & \text{if}\ (a,b)= (m_{11}+m_{21},0) \\
                             \beta_{0,(a,b-1)}({Z_1}) &  \text{otherwise.}
                                \end{cases}$
\end{tabular}

So, by the inductive hypothesis and using Remark \ref{remD},  we have

\begin{tabular}{l}
$\beta_{0,(a,b)}(Z) =\begin{cases} 1 & \text{if}\ (a,b)= (m_{11}+m_{21},0) \\
                      1 & \text{if}\ (a,b-1)\in D_1(Z_1)\\
                      \varphi_{m_{21}+B_Z,B_Z+1}(b-1)  & \text{if}\ (a,b-1)\in D_2(Z_1)\\
                      1+\varphi_{m_{21}+B_Z,B_Z+1}(b-1)  & \text{if}\ (a,b-1)\in D_3(Z_1)\\
                      1 & \text{if}\ (a-1,b)\in D_4(Z_1)\\
                      0 &  \text{otherwise.}
                                \end{cases}$
\end{tabular}

By using Proposition \ref{R2.10+T2.12} we are
done. Consider now $m_{12}> m_{11}$ i.e. $B_Z>0.$ We get

\begin{tabular}{l}
$\beta_{0,(a,b)}(Z) =\begin{cases}
                        \beta_{0,(a,b-1)}({Z_1})+1 & \text{if}\ a+b= m_{21}+ m_{12}  \ \text{and}\ b\le  B \\
                        \beta_{0,(a,b-1)}({Z_1}) &  \text{otherwise}
                                \end{cases}$
\end{tabular}

where, by inductive hypothesis, the bigraded Betti numbers of
degree zero of $R/I_{Z_1}$ are

{\Small\begin{tabular}{l}
$\beta_{0,(a,b-1)}({Z_1})$=$\begin{cases}
                      \varphi_{m_{21}+B_Z,B_Z+1}(b-1)   & \text{if} \ (a,b-1)\in D_2(Z_1)\\
                       1+\varphi_{m_{21}+B_Z,B_Z+1}(b-1) & \text{if} \ (a,b-1)\in D_3(Z_1)\\
                       1& \text{if}\ (a-1,b)\in D_4(Z_1)\\
                       0 & \text{otherwise.}
                                \end{cases}$\\
\\
\end{tabular}}

Therefore, by using Lemma \ref{R2.10+T2.12}  
and Remark \ref{remD}, we are done.
Finally the computation of $\beta_{1,(a,b)}(Z)$ and
$\beta_{2,(a,b)}(Z)$ requires the same procedure as above by using the inductive hypothesis and Lemma \ref{R2.10+T2.12}.
\end{proof}

When the points have the same multiplicity $Z$ is called a homogeneous set of fat points. The graded Betti numbers for a homogeneous set of points $Z$ are given in Theorem \ref{Thmcaso1}. In this case, it became easier to write them, as the next corollary shows.
\begin{corollary}\label{Thom} Let $Z=mP_{11}+mP_{12}+mP_{21}$ be a homogeneous set of points in $\popo.$ Then the bigraded Betti numbers of $I_Z$ are:

\begin{tabular}{l}

$\beta_{0,(a,b)}(Z)=\begin{cases}
                 \varphi_{m+1}(b) &\text{if}\ a+b=2m\\
                 0 & \text{otherwise} \\
                \end{cases}$\\
\\
$\beta_{1,(a,b)}(Z)=\left( \beta_{0,(a,b-1)}(Z) +\beta_{0,(a-1,b)}(Z) -1 \right)_+$\\
$\beta_{2,(a,b)}(Z)=\left( \beta_{0,(a-1,b-1)}(Z) -1
\right)_+$
\end{tabular}

\end{corollary}
\begin{proof} The proof is an immediate consequence of Theorem \ref{Thmcaso1} and Proposition \ref{indphi} since we have

$\beta_{0,(a,b)}(Z)$=$\begin{cases}
       1 +\varphi_{m}(b)&\text{if}\     b\ge m\ \text{and}\         a+b=2m\\
       \varphi_{m}(b)&\text{if}\    b< m\   \text{and}\         a+b=2m\\
       0 & \text{otherwise.}
                                \end{cases}$\\
\end{proof}
\section{The Hilbert Function of $Z$}
In this section we explicitly compute the Hilbert function of the set of fat points $Z.$ 
Recall that the Hilbert function of $Z$ is a numeric function $H_Z:=H_{R/I_Z}:\mathbb{N}^2\to \mathbb{N},$ defined by
$$H_Z(a,b)=\dim_K (R/I_Z)_{(a,b)}= \dim_k R_{(a,b)}-\dim_k I_Z{(a,b)}.$$
The first difference of the Hilbert function is defined as
 $$\Delta H_{Z}(a,b)=H_{Z}(a,b)+H_{Z}(a-1,b-1)-H_{Z}(a-1,b)-H_{Z}(a,b-1).$$
From now on we will write  $(i,j)\le (a,b)$ iff both $i\le a $ and $ j\le b.$

The following are the multigraded version of well known results for standard graded algebras.   
Next lemma shows how we can compute the Hilbert function of $Z$ from a minimal free resolution of $R/I_Z$. 
\begin{lemma}\label{HFlemma}Let $Z$ be a set of fat points in $\popo,$ then	$$H_Z(a,b)=(a+1)(b+1)- \sum_{(i,j)\le (a,b)}(a-i+1)(b-j+1)(\beta_{0,(i,j)}(Z)-\beta_{1,(i,j)}(Z)+\beta_{2,(i,j)}(Z)).$$
\end{lemma}
	\begin{proof} Since a minimal free resolution of $R/I_Z$ $$0\to \bigoplus R(-i,-j)^{\beta_{2,(i,j)}(Z)}\to \bigoplus  R(-i,-j)^{\beta_{1,(i,j)}(Z)}\to \bigoplus  R(-i,-j)^{\beta_{0,(i,j)}(Z)}\to R\to R/I_Z \to 0 $$
	has bigraded morphisms, we get  the following exact sequence of vector spaces
	{\Small $$0\to \bigoplus \left(R(-i,-j)^{\beta_{2,(i,j)}(Z)}\right)_{(a,b)}\to \bigoplus  \left(R(-i,-j)^{\beta_{1,(i,j)}(Z)}\right)_{(a,b)}\to \bigoplus  \left(R(-i,-j)^{\beta_{0,(i,j)}(Z)}\right)_{(a,b)}\to (R)_{(a,b)}\to (R/I)_{(a,b)} \to 0 $$}
	Moreover {\small$$\dim_k \left(\bigoplus  R(-i,-j)^{\beta_{u,(i,j)}(Z)}\right)_{(a,b)}= \sum_{(i,j)\le (a,b)} \dim_k \left(R(-i,-j)^{\beta_{u,(i,j)}(Z)}\right)_{(a,b)}=
		\sum_{(i,j)\le (a,b) } \beta_{u,(i,j)}(Z)(a-i+1)(b-j+1).$$ }
		\end{proof}	

\begin{corollary}\label{HFcor} Let $B_{u,(a,b)}:=\sum_{(i,j)\le (a,b)} \beta_{u,(a,b)}(Z)$ then
	$$\Delta H_{Z}(a,b)=1-B_{0,(a,b)}+B_{1,(a,b)}-B_{2,(a,b)} $$
\end{corollary}
\begin{proof} This follows from Lemma \ref{HFlemma}. 
\end{proof}

In the following two propositions we explicitly compute the first difference  of the Hilbert Function of $Z$  using Corollary \ref{HFcor}. 	
\begin{proposition}\label{HFcaso1}Let  $Z=m_{11}P_{11}+m_{12}P_{12}+m_{21}P_{21}.$ If $m_{11}\le m_{21}$  then 
	$$\Delta H_Z(a,b)=\begin{cases}1 & \text{if}\ a+b< m_{12}+m_{21}\\
								   1-\beta_{0,(a,b)}(Z) & \text{if}\ a+b= m_{12}+m_{21}\\
								   0 & \text{if}\ a+b> m_{12}+m_{21}.\\\end{cases}
	$$ 
	
\end{proposition}
\begin{proof} By Theorem \ref{Thmcaso1}, we have  $\beta_{1,(i,j)}(Z)=\beta_{0,(i-1,j)}(Z)+ \beta_{0,(i,j-1)}(Z)-1$ for $i+j=m_{12}+m_{21}+1$ and zero elsewhere, and  $\beta_{2,(i,j)}(Z)=\beta_{0,(i-1,j-1)}(Z)-1$ if and only if $i+j=m_{12}+m_{21}+2$ (otherwise zero).
	So,  by applying Corollary \ref{HFcor} and a machinery computation  we are done.
	\end{proof}
\begin{proposition}\label{HFcaso2} Let  $Z=m_{11}P_{11}+m_{12}P_{12}+m_{21}P_{21}.$ If $m_{11}> m_{21}$ then 
	$$\Delta H_Z(a,b)\begin{cases} 
	1          & \text{if}\ (a,b)< (i,j)\  \text{for some}\  (i,j)\in \cup D_i(Z)\\
	1-\beta_{0,(a,b)}(Z) & \text{if}\ (a,b)\in \cup D_i(Z)\\
	0          & \text{otherwise.} \end{cases}$$

\end{proposition}
\begin{proof} The proof use the same argument as in Proposition \ref{HFcaso1}.
	\end{proof}
In the last part of these notes we give a criterion to say if an admissible function, $H:\N^2 \to N,$ as introduced in \cite{GMR1992} Definition 2.2,  is the Hilbert function of a set of, at most, three (fat) points on an ACI support. 

\begin{theorem}\label{CharHF}Let $H:\N^2 \to N$ be an admissible function, and let $H(i,j)=\gamma$ for $(i,j)\gg (0,0)$. We denote by $h_{ij}:=\Delta H(i,j)$, moreover we set  $$A_i^{(d)}:=\sum_{j=0}^{d-1} h_{ij},\ \ \ B_j^{(d)}:=\sum_{i=0}^{d-1} h_{ij}$$ and  $\alpha:=\max\{i\ |\ A_i^{(d)}\neq 0\}+1,$ and $\beta:=\max\{j\ |\ B_j^{(d)}\neq 0\}+1.$ Then we have the following cases:  
	
\noindent Case $1)$ There exist $d,d_1,d_2\in \mathbb{N}$ such that if $i<d_1$ and $j<d_2$ then $h_{ij}=1$ iff  $i+j<d.$ 
If $H$ is the Hilbert function of a set of points $Z:=m_{11}P_{11}+m_{12}P_{12}+m_{21}P_{21}$ then  $(m_{11},m_{12},m_{21})$ is the solution of one of the following systems:
$$\left\{\begin{array}{r}
x+y=\alpha\\
x+z=\beta\\
y+z=d\\

\end{array}\right. \ \ \ \ \ \left\{\begin{array}{r}
x+y=\alpha\\
y=\beta\\
y+z=d.\\
\end{array}\right. \ \ \ \ \ \left\{\begin{array}{r}
z=\alpha\\
x+z=\beta\\
y+z=d\\

\end{array}\right. \ \ \ \ \ \left\{\begin{array}{r}
z=\alpha\\
y=\beta\\
{x+1\choose 2}+{y+1\choose 2}+{z+1\choose 2}=\gamma.\\

\end{array}\right.
  $$


\noindent Case $2)$ Assume the first case does not occur.  Then $H$ is not  the Hilbert function of any set of points  $m_{11}P_{11}+m_{12}P_{12}+m_{21}P_{21}.$
\end{theorem} 
\begin{proof}Case (1). The condition $h_{ij}=1$ iff $i+j<d$ ($i<d_1,\ j<d_2$) is always verified for sets of three points on an ACI support (see Theorem \ref{Thmcaso1} and Theorem \ref{Thmcaso2}). In both cases we have $m_{12}+m_{21}=d.$   Moreover, from Theorem 2.12 in \cite{GMR1992}, $\alpha$ and $\beta$ respectively count the maximum number of point on a line of type $(0,1)$ and $(1,0)$  that are respectively $\max\{m_{11}+m_{21},m_{12}\}$ and $\max\{m_{11}+m_{12},m_21\}.$
These conditions give arise to four linear systems:
$$\left\{\begin{array}{r}
x+y=\alpha\\
x+z=\beta\\
y+z=d\\

\end{array}\right. \ \ \ \ \ \left\{\begin{array}{r}
x+y=\alpha\\
y=\beta\\
y+z=d.\\
\end{array}\right. \ \ \ \ \ \left\{\begin{array}{r}
z=\alpha\\
x+z=\beta\\
y+z=d\\

\end{array}\right. \ \ \ \ \ \left\{\begin{array}{r}
z=\alpha\\
y=\beta\\
y+z=d.\\
\end{array}\right.
$$
But the last system is not determined, so we need to replace one equation with ${x+1\choose 2}+{y+1\choose 2}+{z+1\choose 2}=\gamma$, that is the degree of a set of three fat points.

Moreover, from Proposition \ref{HFcaso1} and  Proposition \ref{HFcaso2} we can see that case (2) does not lead to any set of at most three fat points on an ACI support.
\end{proof}

Given an admissible numerical function $H$, Theorem \ref{CharHF} allows us to find the multiplicities of the three points candidate to have as Hilbert function $H$. Supposed to be in Case (1), we construct sets of points with multiplicity as found by solving the systems. Then, by using Theorem \ref{Thmcaso1} and Proposition \ref{HFcaso1} or \ref{Thmcaso2} and Proposition \ref{HFcaso2}, we compute the Hilbert function of these points and hence we compare them with $H.$ The next example shows this procedure. 

\begin{example}Let $H:\N^2\to \N$ be a numerical function such that
	$$\Delta H=\begin{array}{l|cccccccccc}
	& 0 & 1 & 2 & 3 & 4 & 5 & 6& 7&  \cdots \\
	\hline
	0 & 1 &  1 &  1 &  1 &  1 & 1 & 1 & 0& \cdots \\
	1 & 1 &  1 &  1 &  1 &  1 & 0 & 0 & 0& \cdots \\
	2 & 1 &  1 &  1 &  1 &  0 & 0 & 0 & 0& \cdots \\
	3 & 1 &  1 &  1 &  -1 &  0 & 0 & 0 & 0& \cdots \\
	4 & 1 &  1 &  -2 &  0 &  0 & 0 & 0 & 0& \cdots \\
	5 & 1 &  -1 &  0 &  0 &  0 & 0 & 0 & 0& \cdots \\
	6 & 0 &  0 &  0 &  0 &  0 & 0 & 0 & 0& \cdots \\ 
	
	\vdots & \vdots & \vdots & \vdots & \vdots & \vdots & \vdots & \vdots & \vdots & \ddots \\
	\end{array}$$
	
	Note that we are in Case (1) of Theorem \ref{CharHF}, in particular we have $d=6$. 
	Thus $\gamma= \sum \Delta H(i,j)=18$ $\alpha=4$ and $\beta=7.$
	So we get the following systems to solve 
	$$(i)\left\{\begin{array}{r}
	x+y=4\\
	x+z=7\\
	y+z=6\\
	\end{array}\right. \ \ \ \ \ (ii)\left\{\begin{array}{r}
	x+y=4\\
	y=7\\
	y+z=6.\\
	\end{array}\right. \ \ \ \ \ (iii)\left\{\begin{array}{r}
	z=4\\
	x+z=7\\
	y+z=6\\
	\end{array}\right. \ \ \ \ \ (iv)\left\{\begin{array}{r}
	z=4\\
	y=7\\
	x^2+x+18=0.\\
	\end{array}\right.
	$$
Note that $(i), (ii), (iv)$ have not solution in $\N^3$, i.e. $H$ is the Hilbert function of a set of fat points $Z=m_{11}P_{11}+m_{12}P_{12}+m_{21}P_{21}$ if and only if $(m_{11},m_{12},m_{21})=(3,2,4),$ that is the solution of $(iii).$ But from Proposition \ref{HFcaso2} and Theorem \ref{Thmcaso2} we have that $Z=3P_{11}+2P_{12}+4P_{21}$ has the first difference of  the Hilbert function equal to 
$$\Delta H_Z=\begin{array}{l|cccccccccc}
& 0 & 1 & 2 & 3 & 4 & 5 & 6& 7&  \cdots \\
\hline
0 & 1 &  1 &  1 &  1 &  1 & 1 & 1 & 0& \cdots \\
1 & 1 &  1 &  1 &  1 &  1 & 0 & 0 & 0& \cdots \\
2 & 1 &  1 &  1 &  1 &  0 & 0 & 0 & 0& \cdots \\
3 & 1 &  1 &  1 &  -1 &  0 & 0 & 0 & 0& \cdots \\
4 & 1 &  1 &  -1 &  0 &  0 & 0 & 0 & 0& \cdots \\
5 & 1 &  -1 &  0 &  0 &  0 & 0 & 0 & 0& \cdots \\
6 & 0 &  0 &  0 &  0 &  0 & 0 & 0 & 0& \cdots \\ 

\vdots & \vdots & \vdots & \vdots & \vdots & \vdots & \vdots & \vdots & \vdots & \ddots \\
\end{array}$$
so $H\neq H_Z$ and hence $H$ is not the Hilbert function of any set of at most three fat points on an ACI support.
\end{example}





\begin{thebibliography}{99}
\bibitem[F]{F} G. Favacchio, \textit{The Hilbert function of bigraded algebras in $k[\popo].$} Preprint arxiv.org/pdf/1609.06950v2.pdf	
	
\bibitem[FG]{FGu} G. Favacchio, E. Guardo \textit{The minimal free resolution of fat almost complete intersections in $\popo$.} Canadian Journal of Mathematics. Published electronically on December 23, 2016. http://dx.doi.org/10.4153/CJM-2016-040-4

\bibitem[GeGrRo]{GGR1986}A.V. Geramita, D. Gregory, L. Roberts,{\it Monomial ideals and points in projective space}. J. Pure Appl. Algebra {\bf 40} (1986), no. 1, 33--62.


\bibitem[GeMaRo]{GMR1983} A.V. Geramita, P. Maroscia, L.G. Roberts, {\it The Hilbert function of a reduced k-algebra.} J. London Math. Soc. (2) {\bf 28} (1983), no. 3, 443--452.

\bibitem[GiMR]{GMR1992} S. Giuffrida, R. Maggioni, A. Ragusa, {\it On the postulation of 	0-dimensional subschemes on a smooth quadric.} Pacific J. Math. {\bf 155} (1992), no. 2, 251--282.

\bibitem[GVT]{GuVTbook} E. Guardo, A. Van Tuyl,
{\it Aritmetically Cohen-Macaulay Sets of points  in $\mathbb{P}^1
	\times \mathbb{P}^1$.} SpringerBriefs in Mathematics (2015)

\end{thebibliography}
\end{document}